\documentclass[12pt, reqno]{amsart}
\usepackage{amsmath, amsthm, amscd, amsfonts, amssymb, graphicx, color}
\usepackage[bookmarksnumbered, colorlinks, plainpages]{hyperref}
\hypersetup{colorlinks=true,linkcolor=red, anchorcolor=green, citecolor=cyan, urlcolor=red, filecolor=magenta, pdftoolbar=true}

\newtheorem{theorem}{Theorem}[section]
\newtheorem{lemma}[theorem]{Lemma}
\newtheorem{proposition}[theorem]{Proposition}

\theoremstyle{definition}
\newtheorem{definition}[theorem]{Definition}

\theoremstyle{remark}
\newtheorem{remark}[theorem]{Remark}
\numberwithin{equation}{section}
\def\id{{\bf 1}\!\!{\rm I}}

\begin{document}

\setcounter{page}{1}

\title[quantum Markov fields]{construction of a new class of quantum Markov fields}

\author[L. Accardi, F. Mukhamedov \MakeLowercase{and} A. Souissi]{Luigi Accardi,$^1$ Farrukh Mukhamedov,$^2$$^{*}$ \MakeLowercase{and} Abdessatar Souissi$^{3}$}

\address{$^{1}$
Centro Vito Volterra, Universita  di Roma "Tor Vergata", Roma
I-00133, Italy}
\email{\textcolor[rgb]{0.00,0.00,0.84}{accardi@volterra.uniroma2.it}}

\address{$^{2}$Department of Mathematical Sciences, College of Science, The United
Arab Emirates University, P.O. Box 15551, Al Ain, Abu Dhabi, UAE}
\email{\textcolor[rgb]{0.00,0.00,0.84}{farrukh.m@uaeu.ac.ae}}

\address{$^{3.1}$ Department of Mathematics,
Faculty of Sciences of Tunis, University of Tunis El-Manar, 1060
Tunis, Tunisia.}

\address{$^{3.2}$Preparatory Institute for Scientific and Technical
 Studies, La Marsa, Carthage University, Tunisia0}
\email{\textcolor[rgb]{0.00,0.00,0.84}{s.abdessatar@hotmail.fr;
abdessattar.souissi@ipest.rnu.tn}}


\let\thefootnote\relax\footnote{Copyright 2016 by the Tusi Mathematical Research Group.}



\date{Received: xxxxxx; Revised: yyyyyy; Accepted: zzzzzz.
\newline \indent $^{*}$Corresponding author}

\begin{abstract}
In the present paper, we propose a new construction of quantum
Markov fields on arbitrary connected,
 infinite, locally finite graphs. The construction
  is based on a specific tessellation
 on the considered graph, that allows us to express
  the Markov property for the local
 structure of the graph. Our main result concerns the existence
  and uniqueness of quantum Markov field over such graphs.
\end{abstract} \maketitle

\section{Introduction}

One of the basic open problem in quantum probability is to develop a
theory of quantum Markov fields, which are conventionally  quantum
Markov processes with multi-dimensional index set. Here Quantum
Markov fields are noncommutative extensions of the classical Markov
fields (see \cite{D,Geor,Pr}). On the other hand, these quantum
fields can be considered as extensions of quantum Markov chains
\cite{[Ac74f],fannes2} to general graphs.

In \cite{[AcFi01a],[Liebs99]} the first attempts to construct
quantum analogues of classical Markov chains have been carried out.
In \cite{[AcFi01a]} quantum Markov fields were considered over
integer lattices, unfortunately there was not given any non trivial
examples of such fields. In \cite{AOM,fannes}, quantum Markov chains
(fields) on the tree like graphs (like Cayley tree) have been
constructed and investigated, but the proposed construction does not
work for general graphs.

A main aim of the present paper is to provide a construction of new
class of quantum Markov fields on arbitrary connected,
 infinite, locally finite graphs. The construction
  is based on a specific tessellation
 on the considered graphs, it allows us to express
  the Markov property for the local
 structure of the graph. Our main result is the existence
  and uniqueness of quantum Markov field over such graphs. We note that even in the
  classical case, the proposed construction gives  a new ways to define
  Markov fields (see \cite{Spa,[Za83]}).

  \section{Graphs}

Let $G=(V, E)$ be a  ( non-oriented  simple ) graph, that is, $L$ is
a nonempty set and $E$ is identified as a subset of an ordered pairs
of $V$, i.e. $$E\subset \{ \{x,y \} \mid x,y \in E, x\ne y\}$$
Elements of $V$ and $E$ are called, respectively, \textit{vertices}
and \textit{edges}. Two vertices $x$ and $y$ are said to be
\textit{nearest neighbors} if there exist an edge joining them (i.e.
$\{x,y\}\in E$) and we denote them by $x\sim y$. For any vertex
$y\in V$ we denote its nearest neighbors by
\begin{equation}\label{Ny}
N_y := \{x\in V \mid y\sim x\}
\end{equation}
Notice that $x\notin N_x$. The set $\{y\}\cup N_y$ is said to be
\textit{interaction domain} or \textit{plaquette} at $y$. If for
every $x\in V$ one has $|N_x|<\infty$ then the graph is called
\textit{locally finite}. An \textit{edge path } or \textit{walk}
joining two vertices $x$ and $y$ is a finite sequence of edges
$x=x_0\sim x_1\sim\dots x_{d-1}\sim x_d=y$. In this case $d$ is the
length of the edge path. The graph is said to be \textit{connected }
if every two disjoint vertices can be joined by an edge path. In the
sequel, we assume that the graph $G$ is infinite, connected and
locally finite. Note that in this case the set $V$ is automatically
countable. Now for any nonempty $\Lambda\subset V$ we associate its
following parts:
\begin{itemize}
\item \textit{complement}:
\begin{equation} \label{complement}
\Lambda^c := V\setminus \Lambda
\end{equation}
\item \textit{boundary}:
\begin{equation}\label{boundary}
\partial \Lambda := \{x\in \Lambda \; \mid \; \exists y\in \Lambda^c;\quad  x\sim y\}
\end{equation}
\item \textit{interior}:
\begin{equation}\label{interior}
\overset{\circ}{\Lambda} := \Lambda\setminus \partial \Lambda
\end{equation}
\item \textit{external boundary}:
\begin{equation}\label{external}
\vec\partial \Lambda := \{y\in \Lambda^c \; \mid \quad \exists x\in
\Lambda;\quad  x\sim y\}
\end{equation}
\item \textit{closure}:
\begin{equation}\label{closure}
\overline{\Lambda} := \Lambda\cup \vec\partial \Lambda
\end{equation}
\end{itemize}
By $\mathcal F$ we denote a net of all finite subsets of $V$, i.e.
\begin{equation}\label{finite_subets}
\mathcal F := \{ \Lambda\subset V \mid | V|<\infty\}
\end{equation}
where $|\cdot|$ denotes the cardinality of a set.

\section{Tessellations on graphs}\label{tess}

In this section we propose a tessellation on the considered graphs,
which will play a key role in the construction. Therefore, the
resulting quantum Markov field will depend also on the tessellation.

 Fix a "root"  $y_1 \in V$ and define by induction the following sets:
\begin{equation}
   V_{0,1}:=\{ y_1\}
\end{equation}
 Having defined $V_{0,n}$, put
\begin{equation}\label{levels}
V_n := \bigcup_{y \in V_{0,n}} \left( \{y \} \cup N_y \right)
\end{equation}
\begin{equation}
 V_{0,n+1}:=V_{0,n}\cup\vec{\partial} V_n
\end{equation}

Define the following set of vertices:
\begin{equation}
V_0 := \bigcup_{n\geq 1} V_{0,n}.
\end{equation}
From now on, elements of $V_0$ will be called \textit{vertices}, any
other element of $V$ belongs to some plaquette at a certain element
of $V_0$. Notice that with in this construction, for every $n$, the
inner boundary $\partial V_n$ of each $V_{n}$ contain no vertex:
$$
\partial V_{n}\cap V_{0}=\emptyset
$$
Since $|V| = + \infty$ and, by assumption, $V$ is connected, one has
$$
|V_{n+1}| \ge |\overline V_n| + 1\geq |V_n| +2,  \qquad \qquad
|V_{0,n+1}|\ge |V_{0,n}|+1
$$
It follows that, if $\Lambda$ is any finite set,
there exists $N \in \mathbb N$ such that
$$
\Lambda \subseteq V_N
$$
Therefore $\{V_n\}$ is an exhaustive sequence of finite subsets
recovering the all the vertices set $V$.

One can check that
\begin{equation}\label{Vee0}
V_{0} :=\{y_1\}\cup \bigcup_{n\geq 1} \vec\partial V_{n}
\end{equation}
and
\begin{equation}\label{VinftyV}
V=\bigcup_{y\in V_{0}}\{y\}\cup N_y
\end{equation}
\begin{remark}
\begin{enumerate}
\item[(i)] For each $x \in V\setminus V_0 $,
there exists at least one $y \in V_0 $ such that $x$ belongs to
the
plaquette at $y$.\\
\item[(ii)] Each $y \in V_{0}$ belongs to its plaquette (i.e. the plaquette  $\{y\}\cup N_y$) but no other one with center in $V_0$.
\end{enumerate}
\end{remark}
The set $V_0$ given by (\ref{Vee0}) (or equivalently the family $\{
V_{0,n}; \quad n=1,2,\cdots\}$ ) is called \textit{tessellation} on
the graph $G$.

\section{Quantum Markov Fields}\label{alg}

In this section we propose a definition for backward Markov fields,
for the same graph $G=(V,E)$ with the given tessellation $\{V_{0,n}:
\quad n=1, 2, \cdots\}$.

The map
\begin{equation}
x \in V \longrightarrow \mathcal H_x \ \hbox{`` state space on $x$
''}
\end{equation}
defines a bundle on $V$ whose fiber is a finite dimensional Hilbert space
$\mathcal H_x$. Denote $\mathcal A_x:=\mathcal B(\mathcal H_x), x\in
V$. Define for any finite subset $\Lambda\subset V$ the algebra
\begin{equation}\label{AV}
\mathcal A_{\Lambda}:= \bigotimes_{x \in \Lambda}\mathcal A_x.
\end{equation}
then one get on a canonical way, the quasi-local algebra
\begin{equation}\label{A}
\mathcal A_V:= \bigotimes_{x \in V}\mathcal A_x
\end{equation}
defined as the closure of the local algebra
\begin{equation}
\mathcal A_{V,loc}:= \bigcup_{\Lambda\in\mathcal F}\mathcal
A_{\Lambda}.
\end{equation}
where $\mathcal F$ is given by (\ref{finite_subets}).

Analogously, one can define for any subset $\Lambda'\subset V$,
the algebra $\mathcal A_{\Lambda'}:=\bigotimes_{x\in\Lambda'}\mathcal A_x$.\\
Notice that for $\Lambda\subset\Lambda'\subset V$ one can see
$\mathcal A_{\Lambda}$  as $C^*$-subalgebra of $\mathcal
A_{\Lambda'} $
 through the following embedding
\begin{equation}
 \mathcal A_{\Lambda}\equiv \mathcal A_{\Lambda}\otimes \id_{\Lambda'\setminus \Lambda }\subset \mathcal A_{\Lambda'}
 \end{equation}

\begin{definition}
Consider a triplet $C \subset B \subset A$ of unital $C^*$-algebras.
Recall \cite{ACe} that a {\it quasi-conditional expectation} with
respect to the given triplet is a completely positive (CP), unital
linear map $ \mathcal{E} \,:\, A \to B$ such that $ \mathcal{E}(ca)
= c \mathcal{E}(a)$, for all $a\in A,\, c \in C$.
\end{definition}

We give the definition of general of backward quantum Markov field
which is independent of the tessellation.

\begin{definition}\label{defQ-Markov-field}
A state $\varphi$ on $\mathcal A_V$ is said to be \textit{backward
quantum Markov field} if for any sequence $\{\Lambda_n
\}_{n=0}^{\infty} $ of finite subsets of $V$ satisfying
$\Lambda_n\subset\subset\Lambda_{n+1}$, there exists a pair
$(\varphi_{\Lambda_0}, \{E_{\Lambda_{n},
\Lambda_{n+1}}\}_{n=0}^{\infty}\})$ with  $\varphi_{\Lambda_0}$
 is a state on $\mathcal A_{\Lambda_0}$ and  $E_{\Lambda_{n},\Lambda_{n+1}}$ is a
  quasi-conditional expectation with respect to the triplet
 $\mathcal A_{\Lambda_{n}}\subset\mathcal A_{\bar\Lambda_n}\subset \mathcal A_{\Lambda_{n+1}}$
 such that
\begin{equation}\label{bk-M_field}
\varphi =\lim_{n\to \infty} \varphi_{\Lambda_0}\circ E_{\Lambda_0,
\Lambda_1}\circ\dots\circ E_{\Lambda_n, \Lambda_{n+1}}
\end{equation}
where the limit is taken in the weak-*-topology.
\end{definition}

\begin{remark}
In Definition \ref{defQ-Markov-field}, the condition
$\Lambda_n\subset\subset\Lambda_{n+1}$ for every
 $n\in\mathbb N$ implies that $\Lambda_n\uparrow V$ and the limit state obtained by the right side of the equation (\ref{bk-M_field}) is defined on the full algebra $\mathcal A_V$,

 If $\varphi$ is a \textit{backward quantum Markov field} in the sense of Definition \ref{defQ-Markov-field},
  then it satisfy Definition 4.2 of \cite{AOM} for any increasing sequence  $\{\Lambda_n\}_{n=0}^{\infty}$
  of finite subsets of $V$ such that $\bar\Lambda_n=\Lambda_{n+1}$, to get such
  a sequence  of subsets, we consider  $\Lambda_0\in\mathcal F$, and  for $\in n\geq1 $ put
$$\Lambda_n=\bar\Lambda_{n-1}.$$
Clearly one has $\Lambda_n\subset\subset \Lambda_{n+1}$ and
$\Lambda_n \uparrow V$.
  \end{remark}

Now we introduce a class of \textit{backward quantum Markov field}
that depends on the tessellation $\{V_{0,n},\quad  n=1,2,\cdots\}$

\begin{definition}\label{defQ-Markov-field-tess}
A state $\varphi$ on $\mathcal A_V$ is said to be \textit{backward
quantum Markov field} w.r.t. the tessellation  $\{V_{0,n},\quad
n=1,2,\cdots\}$, ( or \textit{$V_0$-backward quantum Markov field})
if for any sequence $\{\Lambda_n \}_{n=0}^{\infty} $ of finite
subsets satisfying
 \begin{equation}\label{tesselation_assumption}
 \Lambda_n\subset\subset\Lambda_{n+1},\quad \vec\partial \Lambda\cap V_0 =\emptyset
 \end{equation}
 there exists a pair $(\varphi_{\Lambda_0}, \{E_{\Lambda_{n},
\Lambda_{n+1}}\}_{n=0}^{\infty}\})$ with  $\varphi_{\Lambda_0}$
 is a state on $\mathcal A_{\Lambda_0}$ and $E_{\Lambda_{n},\Lambda_{n+1}}$ is a
  quasi-conditional expectation with respect to the triplet
 $\mathcal A_{\Lambda_{n}}\subset\mathcal A_{\bar\Lambda_n}\subset \mathcal A_{\Lambda_{n+1}}$
 such that
\begin{equation}\label{bk-M_field-tess}
\varphi=\lim_{n\to \infty} \varphi_{\Lambda_0}\circ E_{\Lambda_0,
\Lambda_1}\circ\dots\circ E_{\Lambda_n, \Lambda_{n+1}}
\end{equation}
where the limit is taken in the weak-*-topology.
\end{definition}

Now we fix the following product state
\begin{equation}\label{phi0}
\varphi^{0}:= \bigotimes_{x\in V}\varphi^{0}_{x}
\end{equation}
on the algebra $\mathcal A_V$, where $\varphi^0_x$ is a state on
$\mathcal A_x$ for every $x\in V$. Denote for $\Lambda\subset V$,
 \begin{equation}
 \varphi^0_{\Lambda}:= \bigotimes_{x\in \Lambda}\varphi_x^0
 \end{equation}
which is the restriction of the state $\varphi^0_V$  to $\mathcal
A_{\Lambda}$.

We aim to construct a quantum Markov field on the algebra $\mathcal
A_V$ through a perturbation of the product state $\varphi_V^0$.

\section{Construction of conditional density amplitudes }

It is well known \cite{ACe} that quasi-conditional expectations are
more convenient than Umegaki conditional expectations (see
definition (\ref{umegakiexp})) to express
  the non-commutative Markov property. In what follows, we will perturb $\varphi$-conditional
  expectations (see \cite{ACe}) to get quasi-conditional expectations using
  a commuting set of operators with the considered tessellation.

For any ordered pair $y\in V_0$ and $x\in N_y$, let be given an
operator
$$
\tilde K_{(x,y)}\in \mathcal A_{\{x,y\}}
$$
such that it is invertible and the $C^*$-subalgebra
\begin{equation}\label{algebra_amplitudes}
\mathcal K=\overline{\{\tilde K_{\{x,y\}}^* \ , \  \tilde
K_{\{x,y\}} \, : \, y\in V_0, \ x\in N_y\}}^{C^*}
\end{equation}
is commutative.

\begin{definition}\label{umegakiexp}
A \textit{Umegaki conditional expectation} is a norm one projection
from a $C^\ast$-algebra onto its $C^\ast$-subalgebra.
\end{definition}

\begin{definition}
Let $\mathcal A_1, \mathcal A_2$ be two $C^\ast$-algebras  with units respectively
 $I_1$ and $ I_2$ and let $\mathcal A=\mathcal A_1\otimes \mathcal A_2$.
An element $K\in\mathcal A$ is called a \textit{conditional density
amplitude} w.r.t. a state $\varphi$ on
  $I_1\otimes \mathcal A_2$, if one has
 \begin{equation}
 \mathbb E^\varphi(K^\ast K)=I_1
 \end{equation}
 where $\mathbb E^\varphi$ is the Umegaki conditional expectation from $\mathcal A$
  onto $\mathcal A_1\otimes I_2$ defined by the linear extension of
  \begin{equation}
  \mathbb E^\varphi(a_1\otimes a_2) = \varphi(I_1\otimes a_2) a_1\otimes
  I_2.
  \end{equation}
An operator $K$ is also called a \textit{conditional density
amplitude } for the $\varphi$-conditional
  expectation $\mathbb E^\varphi$.
\end{definition}

For each $x\in V$ by $\mathbb E^0_{\{x\}^c}$ we denote the Umegaki
conditional expectation from the algebra $\mathcal A$ onto the
algebra $\mathcal A_{\{x\}^c}$ defined on localized elements
$a=\bigotimes_{z\in V} a_{z}= a_{x}\otimes a_{\{x\}^c}$ by:
\begin{equation}\label{umegakix}
\mathbb E^0_{\{x\}^c}(a_x\otimes a_{\{x\}^c})=
\varphi^0_x(a_x)a_{\{x\}^c}.
\end{equation}

One can prove the following facts.

\begin{lemma}\label{cmm}
For every pair of vertices $(x,y)\in V^2$ one has
\begin{equation}
[\mathbb E^0_{\{x\}^c}, \mathbb E^0_{\{y\}^c}]=0
\end{equation}
\end{lemma}

For any $\Lambda\in\mathcal F$ by virtue of Lemma \ref{cmm} we
define
\begin{equation}\label{E0lambda}
\mathbb E^{0}_{\Lambda^c}:= \prod_{x\in\Lambda} \mathbb
E^0_{\{x\}^c}.
\end{equation}

\begin{lemma}\label{E0LamCE}
For any $\Lambda\in\mathcal F$ the map $\mathbb E^{0}_{\Lambda^c}$
given by \eqref{E0lambda} is a Umegaki conditional expectation from
$\mathcal A_V$ onto $\mathcal A_{\Lambda^c}$ such that for
$a_{\Lambda}\in\mathcal A_{\Lambda} a_{\Lambda^c}\in\mathcal
A_{\Lambda^c}$ one has
 \begin{equation}\label{E0express}
\mathbb E^0_{\Lambda^c}(a_{\Lambda}\otimes a_{\Lambda^c})=
\varphi^0_\Lambda(a_{\Lambda})a_{\Lambda^c}
\end{equation}
\end{lemma}

\begin{remark}
The map $\mathbb E^0_{\Lambda^c}$ can be defined, through the
equation (\ref{E0express}), for an arbitrary part (not necessarily finite)
  $\Lambda$ of $V$ and it is still a Umegaki conditional expectation from $\mathcal A_V$
onto $\mathcal A_{\Lambda^c}$
\end{remark}

\begin{proposition}
Let $y\in V_{0},$ the operator
\begin{equation}\label{Bny}
B_{N_y}:=\mathbb E^0_{\{y\}^c}\left(\bigg|\prod_{x\in N_y} \tilde
K_{\{x,y\}}\bigg|^2\right) \in \mathcal A_{N_y}
\end{equation}
is invertible.
\end{proposition}

\begin{proof}
Let us consider $B_{\{y\}\cup N_y}:=\left|\prod_{x\in N_y} \tilde
K_{\{x,y\}}\right|^2\in \mathcal A_{\{y\}\cup N_y}$  and denote its
spectrum by $\sigma(B_{\{y\}\cup N_y})$, which is a closed subset of
the complex field. Since the operator  $B_{\{y\}\cup N_y}$ is
positive definite, then $\sigma(B_{\{y\}\cup N_y})\subset ]0,
\|\widetilde{K}_{\{y\}\cup N_y}\|]$. Since the spectrum does not
contain zero then there is $\varepsilon>0$ such that
$\sigma(B_{\{y\}\cup N_y})\subset[\varepsilon, \|B\|]$, therefore
$B_{\{y\}\cup N_y}\geq \varepsilon \id$.  This yields $\mathbb
E_{\{y\}^c}^0(B)\geq \varepsilon \id$, which means $B_{N_y}=\mathbb
E^0_{\{y\}^c}$ is invertible.
\end{proof}

In the sequel, we assume that for every $y\in V_0$ the operator
$B_{N_y}$ belongs to the commutant $\mathcal K'$ (w.r.t. $\mathcal
A_V$) of the algebra $\mathcal K$ (see \eqref{algebra_amplitudes}).
Note that under this condition the operators $B^{\pm 1/2}_{N_y}$
also belong to $\mathcal K'$.

\begin{lemma} \label{conddensity}
The operator
\begin{equation}\label{kyny}
K_{\{y\}\cup N_y}:= \left(\prod_{x\in N_y}\widetilde{K}_{\{x,y\}}\right)B_{N_y}^{-1/2}
\end{equation}
is a $\varphi^0_{\{y\}}$-conditional density amplitude in the
algebra $\mathcal A_{\{y\}\cup N_y}$.
\end{lemma}

\begin{proof} Using the commutativity of the algebra $\mathcal{K}$ we
obtain
\begin{eqnarray*}
\mathbb E^0_{\{y\}^c}\left(K_{\{y\}\cup N_y}^*K_{\{y\}\cup
N_y}\right)&=
&\mathbb E^0_{\{y\}^c}\left(B_{N_y}^{-1/2}\left(\prod_{x\in N_y}\tilde K_{\{x,y\}}\right)^*\left(\prod_{y\in N_y}\tilde K_{\{x,y\}}\right)B_{N_y}^{-1/2}\right)\\
&=& (B_{N_y}^{-1/2})^*E^0_{\{y\}^c}\left(\prod_{x\in N_y}\tilde K_{\{x,y\}}^*\tilde K_{\{x,y\}}\right)B_{N_y}^{-1/2}\\
&=& (B_{N_y}^{-1/2})^*B_{N_y}B_{N_y}^{-1/2}\\
&=& \id.
\end{eqnarray*}
This completes the proof.
\end{proof}

Now, for each $ \Lambda\in \mathcal F,$  we define
\begin{equation}\label{df-dext0-Lambda}
\vec{\partial}_0\Lambda := \bigcup_{y \in \partial\Lambda\cap
V_0}N_y
\end{equation}
 By construction the family
$$\{K_{\{y\}\cup N_y}^* \ , \ K_{\{y\}\cup N_y} \ : \ x\sim y\in V_0 \}$$
is commutative, therefore the following operator is well defined
\begin{equation}\label{df-K-bar-Lambda}
K_{\Lambda \cup \vec{\partial}_0\Lambda}:=\prod_{y \in \Lambda\cap
V_0} K_{\{y\} \cup N_y}\in \mathcal A_{\Lambda \cup
\vec{\partial}_0\Lambda}\subseteq \mathcal A_{\bar\Lambda} \ \ \
\textrm{for every} \ \ \Lambda \in\mathcal F.
\end{equation}

\begin{remark}\label{rem-K-bar-Lambda}
\begin{itemize}
\item[1.] In general, it is possible that $\mathcal A_{\Lambda \cup \vec{\partial}_0\Lambda}$
is a proper sub-algebra of $\mathcal A_{\bar\Lambda}$. Since, by
construction of the tessellation, the set $\Lambda \cup
\vec{\partial}_0\Lambda$ cannot contain elements of $V_0$.
\item[2.]  If $\Lambda\cap V_0= \emptyset$,
we convent that  $K_{\Lambda \cup \vec{\partial}_0\Lambda}=\id$.
\end{itemize}
\end{remark}
\begin{theorem}\label{K-Lambda-density}
For any $\Lambda\in\mathcal F,$ the operator
$K_{\Lambda\cup\vec\partial_0\Lambda}$ defined by
(\ref{df-K-bar-Lambda}) is a conditional density amplitude for the
 Umegaki conditional expectation $\mathbb E^{0}_{(\Lambda\cap V_0)^c}$.
\end{theorem}
\begin{proof}
 Since the family $\{K_{\{y\}\cup N_y}, K^\ast_{\{y\}\cup N_y}\, \mid \, y\in \Lambda\cap V_0\}$
 is commutative, then one can write
 \begin{equation}
 K_{\Lambda\cup\vec\partial_0\Lambda}^\ast K_{\Lambda\cup \vec\partial_0\Lambda}
 = \prod_{y\in\Lambda\cap V_0} K_{\{y\}\cup N_y}^\ast K_{\{y\}\cup N_y}
 \end{equation}
 and using the following property of the tessellation:
  \textit{for disjoint elements $y$ and $z$ of $V_0$ the plaquette at $y$ does not contain
  $z$}, we conclude that $K_{\{y\}\cup N_y}$ is localized in $\{z\}^c$. Then  by Lemma
  \ref{umegakix} one gets
 $$\mathbb E^0_{\{z\}^c}(K_{\{y\}\cup N_y}^\ast K_{\{y\}\cup N_y}) = K_{\{y\}\cup N_y}^\ast K_{\{y\}\cup N_y}$$
 then after a small iteration, we obtain
$$
\mathbb E^0_{(\Lambda\cap
V_0)^c}\left(K_{\Lambda\cup\vec\partial_0\Lambda}^\ast
K_{\Lambda\cup\vec\partial_0\Lambda}\right) = \prod_{y\in
\Lambda\cap V_0 }\mathbb E^0_{\{y\}^c}( K_{\{y\}\cup N_y}^\ast
K_{\{y\}\cup N_y}).
$$
By Lemma \ref{conddensity}, one has $\mathbb E^0_{\{y\}^c}(
K_{\{y\}\cup N_y}^\ast K_{\{y\}\cup N_y})=\id$, hence one gets
$$\mathbb E^0_{(\Lambda\cap V_0)^c}\left(K_{\Lambda\cup\vec\partial_0\Lambda}^\ast K_{\Lambda\cup\vec\partial_0\Lambda}\right)=\id.$$
\end{proof}

The following auxiliary results can be easily proved.

\begin{lemma}
For every $\Lambda_1\subset\Lambda_2\subset V$, one has:
\begin{equation}
\mathbb E^0_{\Lambda_1}\circ \mathbb E^{0}_{\Lambda_2}= \mathbb
E^{0}_{\Lambda_1}
\end{equation}
\end{lemma}

\begin{lemma}
For $\Lambda, \Lambda'\subset_{fin} V$ with
$\bar{\Lambda}\cap\Lambda'=\emptyset$, one has:
\begin{equation}\label{fact-bar-Lambda1}
K_{(\Lambda\cup\Lambda')\cup \vec{\partial}(\Lambda\cup \Lambda')}
=K_{\Lambda\cup
\vec{\partial_0\Lambda}}K_{\Lambda'\cup\vec{\partial}_0\Lambda'}
\end{equation}
\end{lemma}

\begin{theorem}\label{E0LbLb0}
For $\Lambda_0 \subseteq \bar{\Lambda}_0 \subset \Lambda$, one has:
\begin{enumerate}
\item[(i)] For $z\in V_0 \cap ( \Lambda\setminus \bar\Lambda_{0})$
\begin{equation}\label{Ezc}
\mathbb E^{0}_{\{z\}^c} (K_{\Lambda\cup \vec\partial_0\Lambda}^\ast
a_{\Lambda_0} K_{\Lambda\cup\vec\partial_0\Lambda}) = K_{(\Lambda
\setminus \{z\})\cup\vec\partial_0 (\Lambda \setminus \{z\}) }^\ast
a_{\Lambda_0} K_{(\Lambda \setminus \{z\})\cup\vec\partial_0
(\Lambda \setminus \{z\}) }
\end{equation}
for every $a_{\Lambda_0}\in \mathcal{A}_{\Lambda_0}$;
\item[(ii)]
\begin{equation}\label{pre-MP2}
\mathbb E^{0}_{(\Lambda\setminus \bar\Lambda_0)\cap V_0}
(K_{\Lambda\cup \vec\partial_0 \Lambda}^\ast a_{\Lambda_0}
K_{\Lambda\cup \vec\partial_0\Lambda}) =K_{\bar{\Lambda}_0\cup
\vec\partial_0(\bar\Lambda_0)}^\ast a_{\Lambda_0}
K_{\bar{\Lambda}_0\cup \vec\partial_0(\bar\Lambda_0)}
\end{equation}
for every $a_{\Lambda_0}\in \mathcal{A}_{\Lambda_0}$.
\end{enumerate}
\end{theorem}
\begin{proof} (i) For a general $\Lambda_0$, if $ z \in (\Lambda \setminus \bar{\Lambda}_0)\cap V_0$, then $N_z$ can
intersect $\vec{\partial} \Lambda_0$, but not $\Lambda_0$.
Therefore, $K_{\{z\}\cup N_z}$ and $a_{\Lambda_0}$ are localized on
disjoint parts, so they commute. Due to the commutativity of
 $\{ K_{\{y\} \cup N_y}, K_{\{y\} \cup N_y}^\ast\, \mid y\in \Lambda\cap V_0\} $ it follows from \eqref{fact-bar-Lambda1} that
\begin{eqnarray*}
&&\mathbb E^{0}_{\{z\}^c} (K_{\Lambda\cup \vec\partial_0 \Lambda}^\ast a_{\Lambda_0} K_{\Lambda\cup \vec\partial_0\Lambda})\\
&=&\mathbb E^{0}_{\{z\}^c} (\prod_{y \in \Lambda\cap V_0} K_{\{y\} \cup N_y}^\ast a_{\Lambda_0}  \prod_{y \in \Lambda\cap V_0}K_{\{y\} \cup N_y})\\
&=& \mathbb E^{0}_{\{z\}^c}\left((K_{\{z\} \cup N_z}^\ast K_{\{y\} \cup N_y})\times (K_{(\Lambda\setminus\{z\})\cup \vec\partial_0 (\Lambda\setminus\{z\})}^\ast a_{\Lambda_0}K_{(\Lambda\setminus\{z\})\cup \vec\partial_0 (\Lambda\setminus\{z\})})\right)\\
&=&\mathbb E^{0}_{\{z\}^c} \left(K_{\{z\} \cup N_z}^\ast K_{\{y\}
\cup N_y}\right) K_{(\Lambda\setminus\{z\})\cup \vec\partial_0
(\Lambda\setminus\{z\})}^\ast
a_{\Lambda_0}K_{(\Lambda\setminus\{z\})\cup \vec\partial_0
(\Lambda\setminus\{z\})}
\end{eqnarray*}
and by Lemma \ref{conddensity} one has
 $$\mathbb E^{0}_{\{z\}^c} \left(K_{\{z\} \cup N_z}^\ast K_{\{y\} \cup N_y}\right)=\id$$
Hence, we get
$$\mathbb E^{0}_{\{z\}^c} (K_{\Lambda\cup \vec\partial_0 \Lambda}^\ast a_{\Lambda_0} K_{\Lambda\cup \vec\partial_0\Lambda})=
K_{(\Lambda\setminus\{z\})\cup \vec\partial_0
(\Lambda\setminus\{z\})}^\ast
a_{\Lambda_0}K_{(\Lambda\setminus\{z\})\cup \vec\partial_0
(\Lambda\setminus\{z\})}$$

(ii)  Iterating the procedure of (\ref{Ezc}) to cover all $z\in
(\Lambda \setminus \bar{\Lambda}_0)\cap V_0$ one finds
\begin{eqnarray*}
\mathbb E^{0}_{(\Lambda\setminus \Lambda_0)\cap V_0}
\left(K_{\Lambda\cup \vec\partial_0 \Lambda}^\ast a_{\Lambda_0}
K_{\Lambda\cup \vec\partial_0\Lambda}\right) &=&\bigg(\prod_{z\in
(\Lambda\setminus \bar\Lambda_0)\cap V_0 }\mathbb
E^{0}_{\{z\}^c}\bigg) \left(K_{\Lambda\cup \vec\partial_0
\Lambda}^\ast a_{\Lambda_0} K_{\Lambda\cup
\vec\partial_0\Lambda}\right)\\[2mm]
&=&K_{\bar{\Lambda}_0\cup \vec\partial_0(\bar\Lambda_0)}^\ast
a_{\Lambda_0} K_{\bar{\Lambda}_0\cup \vec\partial_0(\bar\Lambda_0)}
\end{eqnarray*}
This completes the proof.
\end{proof}

\begin{remark}
Keeping the notations of Theorem \ref{E0LbLb0}, if $\vec\partial
\Lambda_0\cap V_0 =\emptyset$ then using the same argument, one gets
$$\mathbb E^0_{\Lambda\setminus\Lambda_0}(K_{\Lambda\cup\vec\partial_0\Lambda}^\ast a_{\Lambda_0}K_{\Lambda\cup\vec\partial_0\Lambda}) =
K_{\Lambda_0\cup\vec\partial_0\Lambda_0}^\ast
a_{\Lambda_0}K_{\Lambda_0\cup\vec\partial_0\Lambda_0} $$ for every
$a_0\in\mathcal A_{\Lambda_0}$.
\end{remark}

\section{Main result}

In this section, we prove a main result of the paper. First we need
an auxiliary result.

\begin{proposition}
Let $\Lambda_1, \Lambda_2\in\mathcal F$ with
$\Lambda_1\subset\subset\Lambda_2$. Define
\begin{equation}\label{quasiexp12}
E_{\Lambda_1,\Lambda_2}(a) = \mathbb E^0_{(\Lambda_2\setminus\bar\Lambda_1)^c}
\left(K_{(\Lambda_2\setminus\bar\Lambda_1)\cup\vec\partial_0 (\Lambda_2\setminus\bar\Lambda_1)}^\ast
 a K_{(\Lambda_2\setminus\bar\Lambda_1)\cup\vec\partial_0 (\Lambda_2\setminus\bar\Lambda_1)}\right)
\end{equation}
for $a\in \mathcal A_V$. Then $E_{\Lambda_1,\Lambda_2}$ is a
quasi-conditional expectation w.r.t. the triplet $\mathcal
A_{\Lambda_1}\subset \mathcal A_{\bar\Lambda_1} \subset\mathcal
A_{\Lambda_2}$.
\end{proposition}

\begin{proof}
The map $E_{\Lambda_1,\Lambda_2}$ is clearly linear and valued in
$\mathcal A_{\bar\Lambda_1}$.

\textbf{Unitality}: using commutativity of the family $\{ \mathbb
E_{\{z\}^c}\,\mid z\in (\Lambda_2\setminus\Lambda_1)\cap V_0\}$ (by
Lemma \ref{cmm}), one can write
 \begin{equation}\label{decompose-E-Lb12}
 \mathbb E^0_{(\Lambda_2\setminus\Lambda_1)^c}= \mathbb E^{0}_{((\Lambda_2\setminus\bar\Lambda_1)\cap V^c_0 )^c}\circ \mathbb E^{0}_{((\Lambda_2\setminus\Lambda_1)\cap V_0 )^c}
 \end{equation}
 and using Theorem \ref{K-Lambda-density} for $\Lambda =\Lambda_2\setminus\bar \Lambda_1$
 we obtain
 $$\mathbb E^{0}_{((\Lambda_2\setminus\bar\Lambda_1)\cap V_0 )^c}
 \left(K_{(\Lambda_2\setminus\bar\Lambda_1)\cup\vec\partial_0(\Lambda_2\setminus\bar\Lambda_1)}^\ast K_{(\Lambda_2\setminus\bar\Lambda_1)\cup\vec\partial_0(\Lambda_2\setminus\bar\Lambda_1)}\right)= \id$$
 then using (\ref{decompose-E-Lb12}) one finds
 $$\mathbb E^0_{(\Lambda_2\setminus\bar\Lambda_1)^c}
 \left(K_{(\Lambda_2\setminus\bar\Lambda_1)\cup\vec\partial_0(\Lambda_2\setminus\bar\Lambda_1)}^\ast K_{(\Lambda_2\setminus\bar\Lambda_1)\cup\vec\partial_0(\Lambda_2\setminus\bar\Lambda_1)}\right)
 = \mathbb E^{0}_{((\Lambda_2\setminus\bar\Lambda_1)\cap V^c_0 )^c}(\id) =\id$$
hence, $$E_{\Lambda_1,\Lambda_2}(\id)=\id.$$

\textbf{Complete positivity:}  One can check that for any $y\in V_0$
the map
$$a\mapsto E_{\{y\}^c} (a):= \mathbb E^0_{\{y\}^c}(K^*_{\{y\}\cup N_y}a K_{\{y\}\cup N_y})$$
is completely positive. Now using the commutativity of the set $\{
K_{\{y\}\cup N_y}, y\in V_0\}$ one gets
$$E_{\Lambda_1,\Lambda_2} = \mathbb E^0_{(\Lambda_2\setminus\bar\Lambda_1)\cap V_0}
\circ \bigg(\prod_{y\in (\Lambda_2\setminus\bar\Lambda_1)\cap V_0}
E_{\{y\}^c} \bigg).$$ Hence  $E_{\Lambda_1, \Lambda_2}$ is the
composition of completely positive maps, therefore, it is so.

Let $a\in \mathcal A_{\Lambda_2}, c\in\mathcal A_{\Lambda_1}$, while
$K_{(\Lambda_2\setminus\Lambda_1)\cup\vec\partial_0(\Lambda_2\setminus\Lambda_1)}^\ast
\in \mathcal A_{\bar(\Lambda_2\setminus\Lambda_1)}$ then it commutes
with $c$, then using the fact that
$$\mathbb E^{0}_{(\Lambda_2\setminus\Lambda_1)^c}(cd)=c\mathbb E^{0}_{(\Lambda_2\setminus\Lambda_1)^c}(d)$$
 for every $d\in\mathcal A$, one gets
\begin{eqnarray*}
E_{\Lambda_1, \Lambda_2}(ca)&=& \mathbb E^{0}_{(\Lambda_2\setminus\bar\Lambda_1)^c}
 \left(K_{(\Lambda_2\setminus\bar\Lambda_1)\cup\vec\partial_0(\Lambda_2\setminus\bar\Lambda_1)}^\ast ca
 K_{(\Lambda_2\setminus\bar\Lambda_1)\cup\vec\partial_0(\Lambda_2\setminus\bar\Lambda_1)}\right)\\
 &=&\mathbb E^{0}_{(\Lambda_2\setminus\bar\Lambda_1)^c}
 \left(c K_{(\Lambda_2\setminus\bar\Lambda_1)\cup\vec\partial_0(\Lambda_2\setminus\bar\Lambda_1)}^\ast a
 K_{(\Lambda_2\setminus\bar\Lambda_1)\cup\vec\partial_0(\Lambda_2\setminus\bar\Lambda_1)}\right)\\
 &=&  c \mathbb E^{0}_{(\Lambda_2\setminus\bar\Lambda_1)^c}
 \left(K_{(\Lambda_2\setminus\bar\Lambda_1)\cup\vec\partial_0(\Lambda_2\setminus\bar\Lambda_1)}^\ast a
 K_{(\Lambda_2\setminus\bar\Lambda_1)\cup\vec\partial_0(\Lambda_2\setminus\bar\Lambda_1)}\right)\\
 &=&c E_{\Lambda_1, \Lambda_2}(a).
\end{eqnarray*}
Hence, $E_{\Lambda_1, \Lambda_2}$ is a quasi-conditional expectation
w.r.t. the given triplet. This completes the proof.
\end{proof}

Now we pass to our main result.

\begin{theorem}\label{main}
For each $\Lambda\in \mathcal F$ define the state
$\widetilde{\varphi}_{\Lambda}$ on $\mathcal A$ by
\begin{equation}
\widetilde{\varphi}_{\Lambda}(a):=
\varphi^{0}(K_{\Lambda\cup\vec\partial_0\Lambda}^\ast a
K_{\Lambda\cup\vec\partial_0\Lambda})
\end{equation}
Then the net $\{\widetilde{\varphi}_{\Lambda}\}_{\Lambda\in\mathcal
F}$ converges in the weak-*-topology, moreover the limiting state
$\varphi$ is a backward Markov field on $\mathcal A_V$ w.r.t. the
tessellation $V_0$.
\end{theorem}

\begin{proof} First we prove the existence of the limit. Due to the density argument, it is sufficient to establish the existence
of the limit in the local algebra $\mathcal A_{V, loc}$.

Let $a\in\mathcal A_{V, loc}$ then $a\in\mathcal A_{\Lambda_0}$ for
some $\Lambda_0\in\mathcal F$. For $\Lambda\in \mathcal F$ with
$\Lambda_0\subset\subset \Lambda$, we have
\begin{eqnarray*}
\widetilde{\varphi}_{\Lambda}(a)
&=& \varphi^{0}\left(K_{\Lambda\cup\vec\partial_0\Lambda}^\ast a K_{\Lambda\cup\vec\partial_0\Lambda}\right)\\
&=& \varphi^{0}\circ \mathbb E_{(\Lambda\setminus\bar\Lambda_0)^c}
\left(K_{\Lambda\cup\vec\partial_0\Lambda}^\ast a K_{\Lambda\cup\vec\partial_0\Lambda}\right)
\end{eqnarray*}
and by Theorem \ref{E0LbLb0} one gets
$$\mathbb
E_{(\Lambda\setminus\bar\Lambda_0)^c}
\left(K_{\Lambda\cup\vec\partial_0\Lambda}^\ast a
K_{\Lambda\cup\vec\partial_0\Lambda}\right) =
K_{\bar\Lambda_0\cup\vec\partial_0\bar\Lambda_0}^\ast a
K_{\bar\Lambda_0\cup\vec\partial_0\bar\Lambda_0},$$ so
\begin{eqnarray*}
\widetilde{\varphi}_{\Lambda}(a) &= &\varphi^{0}(K_{\bar\Lambda_0\cup\vec\partial_0\bar\Lambda_0}^\ast
 a K_{\bar\Lambda_0\cup\vec\partial_0\bar\Lambda_0})\\
&=& \widetilde{\varphi}_{\bar\Lambda_0}(a).
\end{eqnarray*}
As $\Lambda\to V$, we find that $\Lambda_0\subset\subset\Lambda$ up
to some order, hence the net
$\{\widetilde{\varphi}(a)\}_{\Lambda\in\mathcal F;
\Lambda_0\subset\subset \Lambda}$ is stationary. This means that
 \begin{equation}
 \lim_{\Lambda\to V; \Lambda_0\subset\subset \Lambda}\widetilde{\varphi}_{\Lambda}(a) = \varphi_{\bar\Lambda_0}(a) =: \varphi(a)
 \end{equation}
therefor the limit exist on the local algebra, and yet it exists on
the full algebra $\mathcal A_V$.

Now we establish that the state $\varphi$ is a quantum Markov field.

Let $\{\Lambda_n \ |\  n\in \mathbb N\}_{n\in \mathbb N}$ be a
family of subset of $\mathcal F$ satisfying
$$\Lambda_n\subset\subset \Lambda_{n+1}, \quad  \vec\partial \Lambda_n\cap V_0 =\emptyset$$
Let  $E_{\Lambda_n, \Lambda_n+1}$ be given by \eqref{quasiexp12}.
Then,  for $a\in\mathcal A_{V_n} $, we have
\begin{eqnarray*}
&&\widetilde\varphi_{\Lambda_{n}}\circ E_{\Lambda_{n}, \Lambda_{n+1}}(a)\\
&=& \varphi^0\left(K_{\Lambda_{n}\cup\vec\partial_0\Lambda_n}^\ast E_{\Lambda_n,\Lambda_{n+1}}(a) K_{\Lambda_n\cup\vec\partial_0\Lambda_n}\right)\\
&=& \varphi^0\left(K_{\Lambda_n\cup\vec\partial_0\Lambda_n}^\ast
 \mathbb E^0_{(\Lambda_{n+1}\setminus \bar\Lambda_{n})^c}\left(K_{(\Lambda_{n+1}\setminus\bar\Lambda_{n})\cup\vec\partial_0 (\Lambda_{n+1}\setminus\bar\Lambda_{n})}^\ast a K_{(\Lambda_{n+1}\setminus\bar\Lambda_{n})\cup\vec\partial_0 (\Lambda_{n+1}\setminus\bar\Lambda_{n})}\right)
 K_{\Lambda_{n}\cup\vec\partial_0\Lambda_{n}}\right)\\
 \end{eqnarray*}
Since  $K_{\Lambda_{n}\cup\vec\partial_0\Lambda_{n}}\in \mathcal
A_{\bar\Lambda_n}
 \subset \mathcal A_{(\Lambda_{n+1}\setminus\bar\Lambda_{n})^c}$ and
  $\mathbb E^0_{(\Lambda_{n+1}\setminus\bar\Lambda_{n})^c}$ is a Umegaki
   conditional expectation from $\mathcal A_V$ onto $\mathcal A_{(\Lambda_{n+1}\setminus\bar\Lambda_{n})^c}$
  then one finds
\begin{eqnarray*}
&&\widetilde\varphi_{\Lambda_n}\circ E_{\Lambda_n, \Lambda_{n+1}}(a)\\
&=& \varphi^0 \mathbb E^0_{(\Lambda_{n+1}\setminus
\bar\Lambda_{n})^c}\left(K_{\Lambda_{n}\cup\vec\partial_0\Lambda_{n}}^\ast
K_{(\Lambda_{n+1}\setminus\bar\Lambda_{n})\cup\vec\partial_0
 (\Lambda_{n+1}\setminus\bar\Lambda_{n})}^\ast a K_{(\Lambda_{n+1}\setminus\bar\Lambda_{n})\cup\vec\partial_0
  (\Lambda_{n+1}\setminus\bar\Lambda_n)} K_{\Lambda_{n}\cup\vec\partial_0\Lambda_{n}}\right)
  \end{eqnarray*}
  and by the assumption (\ref{tesselation_assumption}) one has
  $$\vec\partial \Lambda_{n}\cap V_0=\emptyset$$
  then $\bar\Lambda_n\cap V_0=\Lambda_n\cap V_0$ and
  $$K_{\Lambda_n\cup\vec\partial_0\Lambda_n}= \prod_{y\in \Lambda_n\cap V_0}K_{\{y\}\cup N_y}=
  \prod_{y\in\bar\Lambda_n\cap V_0}K_{\{y\}\cup N_y}= K_{\bar\Lambda_n\cup\vec\partial_0\bar\Lambda_n}.$$
From Lemma \ref{fact-bar-Lambda1} it follows that
  $$K_{\Lambda_{n+1}\cup\vec\partial_0\Lambda_{n+1}} = K_{\bar\Lambda_{n}\cup\bar\Lambda_1}K_{(\Lambda_{n+1}\setminus\bar\Lambda_{n})\cup\vec\partial_0(\Lambda_{n+1}\setminus\bar\Lambda_{n})}$$
  then we obtain
  $$\widetilde\varphi_{\Lambda_n}\circ E_{\Lambda_n, \Lambda_{n+1}}(a)
  =\varphi^{0}\circ \mathbb E^{0}_{(\Lambda_{n+1}\setminus\bar\Lambda_{n})^c}(K_{\Lambda_{n+1}\cup\vec\partial_0\Lambda_{n+1}}^\ast a K_{\Lambda_{n+1}\cup\vec\partial_0\Lambda_{n+1}})$$
 Hence, by construction one gets
$$\varphi^{0}_V  = \varphi^{0}_V \circ \mathbb E_{(\Lambda_{n+1}\setminus\bar\Lambda_{n})^c}$$
so
\begin{equation}\label{projetiphin}
\widetilde{\varphi}_{\Lambda_n}\circ E_{\Lambda_{n}, \Lambda_{n+1}}(a) =
\varphi^0_V(K_{\Lambda_{n+1}\cup\vec\partial_0\Lambda_{n+1}}^\ast a K_{\Lambda_{n+1}\cup\vec\partial_0\Lambda_{n+1}})= \widetilde{\varphi}_{\Lambda_{n+1}} (a)
\end{equation}
Now iterating the equation \eqref{projetiphin}, we obtain
$$\widetilde\varphi_{n}= \widetilde\varphi_{\Lambda_0}\circ E_{\Lambda_0,\Lambda_1}\circ\dots\circ E_{\Lambda_{n-1}, \Lambda_n}$$
therefore
$$\varphi_V= \lim   \varphi_{\Lambda_0}\circ E_{\Lambda_0,\Lambda_1}\circ\dots\circ E_{\Lambda_{n-1}, \Lambda_n}$$
where $\varphi_{\Lambda_0}=
{\widetilde\varphi_{\Lambda_0}}\lceil_{\mathcal A_{\Lambda_0}}$.
This completes the proof.
\end{proof}

The provided construction allows us to produce a lot of interesting
examples of quantum Markov fields on arbitrary connected,
 infinite, locally finite graphs. Note that the construction
  is based on a specific tessellation
 on the considered graph, it allows us to express
  the Markov property for the local
 structure of the graph. We note that even in the
  classical case, the proposed construction gives other ways to define
  Markov fields different to the existing ones (see \cite{Spa}).
This construction opens new perspectives in the theory of phase
transitions in the scheme of quantum Markov fields (comp.
\cite{MBSo}).

\bibliographystyle{amsplain}

\end{document}